\documentclass[a4paper, 10pt, conference]{ieeeconf}      

\IEEEoverridecommandlockouts                              
\usepackage{graphics} 
\usepackage{epsfig} 
\usepackage{mathptmx} 
\usepackage{times} 
\usepackage{amsmath} 
\usepackage{amssymb}  

\usepackage[mathcal]{euscript} 
\usepackage{mathtools}
\usepackage{mathrsfs}
\usepackage{bm}

\usepackage{graphicx}
\graphicspath{{Figures/}} 
\usepackage{caption} 
\usepackage{subfig} 

\usepackage{url}

\usepackage{xcolor}
\definecolor{USred}{rgb}{0.74,0.1,0.1}
\definecolor{USblue}{rgb}{0.2,0.2,0.7}
\definecolor{green1}{cmyk}{0.82,0,1,0.3}

\definecolor{Royalblue}{cmyk}{1,0.30,0.2,0.2}

\newcommand{\Zs}{\mathbb{Z}}
\newcommand{\Es}{\mathbb{E}}
\newcommand{\Rs}{\mathbb{R}}

\newcommand{\al}[1]{\begin{align} #1\end{align}}
\newcommand{\nn}{\nonumber}
\DeclareMathOperator{\tr}{tr}

\DeclareRobustCommand{\vect}[1]{
	\ifcat#1\relax
	\boldsymbol{#1}
	\else
	\mathbf{#1}
	\fi}

\newtheorem{proposition}{Proposition}

\newcommand{\xb}{\mathbf{x}}
\newcommand{\yb}{\mathbf{y}}
\newcommand{\tb}{\mathbf{t}}
\newcommand{\ssb}{\mathbf{s}}
\newcommand{\hb}{\mathbf{h}}
\newcommand{\eb}{\mathbf{e}}
\newcommand{\ub}{\mathbf{u}}
\newcommand{\vb}{\mathbf{v}}
\newcommand{\tth}{\vartheta}
\newcommand{\thb}{\bm{\vartheta}}
\newcommand{\ab}{\bm{\alpha}}
\newcommand{\bb}{\bm{\beta}}
\newcommand{\gb}{\bm{\gamma}}
\newcommand{\zzb}{\mathbf{z}}
\newcommand{\Pc}{\mathcal{P}}
\newcommand{\Nb}{\mathbf{N}}
\newcommand{\ip}[2]{\langle #1,#2 \rangle}
\newcommand{\Ts}{\mathbb{T}}
\newcommand{\dm}{\mathrm{d}\mu}
\newcommand{\Sc}{\mathcal{S}}
\newcommand{\lb}{\bm{\ell}}

\newcommand{\zb}{\bm{\zeta}}
\newcommand{\Ns}{\mathbb{N}}
\title{\LARGE \bf Optimal Transport between Gaussian random fields}

\author{Mattia Zorzi 
\thanks{}
\thanks{M. Zorzi is with the Department of Information Engineering, University of Padova, Padova, Italy; email:	 
	   {\tt\small zorzimat@dei.unipd.it}.}%
\thanks{This work was partially supported by the SID project ``A Multidimensional and Multivariate Moment Problem Theory for Target Parameter Estimation in Automotive Radars'' (ZORZ\_SID19\_01) funded by the Department of Information Engineering of the  University of Padova.}%
}

\begin{document}

\maketitle
\thispagestyle{empty}
\pagestyle{empty}

\begin{abstract} We consider the optimal transport problem between zero mean Gaussian stationary random fields both in the aperiodic and periodic case. We show that the solution corresponds to a weighted Hellinger distance between the multivariate and multidimensional power spectral densities of the random fields. Then, we show that  such a distance defines a geodesic, which depends on the weight function, on the manifold of the  multivariate and multidimensional power spectral densities.
\end{abstract}


\section{Introduction} 
The Optimal Transport Problem (OTP) aims in minimizing the effort to transport one nonnegative measure to another nonnegative measure according to a cost  of moving mass from a point to another one.  
This problem has been formulated by Kantorovitch \cite{kantorovich1942translocation} and in the recent years it has been 
used for deriving new distances between covariance matrices and spectral densities, \cite{knott1984optimal,georgiou2008metrics,elvander2018interpolation,ning2014matrix,chen2017matricial}. In particular, in \cite{zorzi2020optimal} it has been shown that the OTP  between Gaussian stationary stochastic processes leads to weighted Hellinger distance between multivariate and unidimensional power spectral densities. The latter distance is a generalization of the Hellinger distance introduced in \cite{Hellinger_Ferrante_Pavon,ramponi2009globally}.

Distances between spectral densities play a fundamental 
role in spectral analysis. Indeed, the latter can be used in order to design high resolution spectral estimators \cite{A_NEW_APPROACH_BYRNES_2000,georgiou2003kullback,BYRNES_GUSEV_LINDQUIST_RATIONAL_COV_EXT,karlsson2013uncertainty,ALPHA} as well the multivariate extensions \cite{FERRANTE_TIME_AND_SPECTRAL_2012,BETAPRED,BETA,DUAL,ZORZI2019108516,ZORZI2020109053}. These methods have been extended to: 1) stationary 
(i.e. homogeneous) random fields which are characterized by  multidimensional power spectral densities  \cite{georgiou2006relative,SIAM_IM_COMPR,ringh2018multidimensional,ZFKZ2019fusion}; 2) stationary periodic random fields which are characterized by multidimensional power spectral densities whose domain is constituted by a finite number of points \cite{ringh2015multidimensional,zhu2019m}. It is worth noting that in the unidimensional case, the latter case boils down to the so called reciprocal processes,  
\cite{levy1990modeling,CFPP-11,lindquist2013circulant,lindquist2013multivariate, RecS}.

The aim of this paper is to extend the results in \cite{zorzi2020optimal} to Gaussian stationary aperiodic/periodic random fields. More precisely, we formulate the OTP and we show that the corresponding solution is a suitable weighted Hellinger distance between multivariate and multidimensional spectral densities. Moreover, we show this  distance defines a geodesic on the manifold of the multidimensional power spectral densities.  The latter can be used in order to perform spectral morphing \cite{NING_2013} for describing a Gaussian random field whose description slowly varies over time.

The outline of the paper is the  following. In Section \ref{sec_1} we introduce the OTP for Gaussian random fields. In Section \ref{sec_2} we introduce the OTP for Gaussian periodic random fields. Section \ref{sec_3} regards the spectral morphing problem and in Section \ref{sec_4} we present a numerical example. In Section \ref{sec_4b} we discuss the general case, i.e. the Gaussian assumption is not required. Finally, some conclusions are drawn in Section  \ref{sec_5}.

{\em Notation:} $\Rs$, $\Zs$, $\Ns$ denote the set of real, integer and natural numbers, respectively. Given two vectors $\ab$ and $\thb$ of the same dimension, then $\ip{\ab}{\thb}$ denotes their inner product. Let $A$ be an Hermitian matrix, then $A>0$ ($A\geq 0$) means that $A$ is positive (semi)definite; $A^*$ denotes its transposed and conjugate. Moreover, we will consider the Euclidean norms $\|A\|:=\sqrt{\tr(A^*A)}$ and $ \|A\|_W:=\sqrt{\tr(A^*WA)}$ with $W=W^*>0$. Given a function $\Phi(e^{j\thb})$ with $\Ts^d:=[0,2\pi]^d$, such that $\Phi(e^{j\thb})=\Phi(e^{j\thb})^*$, then $\Phi>0$ ($\Phi\geq 0$) means that $\Phi(e^{j\thb})>0$ ($\Phi(e^{j\thb})\geq 0$) for any $\thb\in \Ts^d$. $\ell_1^{m\times m}(\Zs^d)$ is the space of sequences $\hb:=\{\,H_\tb, \; \tb\in\Zs^d\, \}$, with $H_\tb\in\Rs^{m\times m}$, which are absolutely summable. Given two sequences $\hb$ and $\vb$, then $\hb\star \vb$ denotes the discrete convolution operation.

  \section{OTP between random fields}\label{sec_1}
Consider two jointly Gaussian stationary random fields $\xb=
\{\xb_\tb, \; \tb\in \Zs^d\}$ and $\yb=\{\yb_\tb,\; \tb\in\Zs^d\}$ having zero mean and taking values in $\Rs^m$. It is worth noting that the index $\tb=(t_1,t_2,\ldots t_d)$ has dimension $d$. These random fields are completely characterized by the finite dimensional probability density functions 
$$p_\xb(\xb_\tb,\xb_\ssb; \tb, \ssb), \; \; p_\yb(\yb_\tb,\yb_\ssb; \tb, \ssb)$$
 with $\tb,\ssb\in\Zs^d$, while the corresponding joint random field is completely characterized by the finite dimensional probability density 
$$p_{\xb,\yb}(\xb_\tb,\xb_\ssb,\yb_\ub,\yb_\vb; \tb, \ssb,\ub, \vb)$$ with $\tb,\ssb,\ub,\vb\in\Zs^d$.

We consider the following optimal transport problem 
\al{\label{pb_trasp}d(p_\xb,p_\yb)^2= \underset{p_{\xb,\yb}\in \Pc}{\inf} \{\Es[\|  \xb_\tb- \yb_\tb\|^2]  \hbox{ s.t. (\ref{cond_marg1})-(\ref{cond_marg2}) hold}  \} }
where
 \al{\label{cond_marg1}\int_{\Rs^m}\int_{\Rs^m} p_{\xb,\yb}(\xb_\tb, \xb_\ssb,\yb_\ub, & \yb_\vb;\tb,\ssb,\ub,\vb)\mathrm d \yb_\ub \mathrm d \yb_\vb\nn\\ &= p_{\xb}(\xb_\tb, \xb_\ssb;\tb,\ssb),\; \; \tb,\ssb\in\mathbb Z^d \\
\label{cond_marg2} \int_{\Rs^m}\int_{\Rs^m} p_{\xb,\yb}(\xb_\tb,\xb_\ssb,\yb_\ub ,& \yb_\vb;\tb,\ssb,\ub,\vb)\mathrm d \xb_\tb \mathrm d \xb_\ssb\nn\\ &= p_{\yb}(\yb_\ub,\yb_\vb;\ub,\vb),\; \; \ub,\vb\in\mathbb Z^d}
and $\Pc$ is the set of Gaussian joint probability densities $p_{\xb,\yb}$. In plain words, the above problem represents the optimal transport between Gaussian random fields $\xb$ and $\yb$ and the transportation cost is the variance of $\eb:=\xb-\yb$ which can be understood as the discrepancy random field.

Since the joint random field is Gaussian, it is completely characterized by its covariance field
$$ R_\tb=R_{-\tb}^T:=\Es\left[\left[\begin{array}{c}\xb_{\tb+\ssb} \\ \yb_{\tb+\ssb}\end{array}\right] \left[\begin{array}{cc}\xb_\ssb^T & \yb_\ssb^T\end{array}\right]\right], \; \; \tb\in\Zs^d$$
or, equivalently, by its discrete-time multidimensional Fourier transform
\al{\label{def_Phi_xy}\Phi(e^{j\thb}):=\sum_{\tb\in\Zs^d} R_\tb e^{-j \ip{\thb}{\tb}}}
where $\thb=(\tth_1,\tth_2,\ldots, \tth_d)\in \Ts^d$ and it represents the power spectral density of the joint process. Partitioning (\ref{def_Phi_xy}) in a conformable way with respect to $\xb$ and $\yb$, we obtain:
$$ \Phi(e^{j\thb})= \left[\begin{array}{cc} \Phi_{\xb}(e^{j\thb}) &  \Phi_{\xb,\yb}(e^{j\thb}) \\ \Phi_{\yb,\xb}(e^{j\thb}) & \Phi_{\yb}(e^{j\thb}) \end{array}\right]$$ where $\Phi_{\xb}$ and $\Phi_{\yb}$ are the power spectral densities of $\xb$ and $\yb$, respectively.

Since $p_{\xb,\yb}$ and $\Phi$ represent two equivalent descriptions of the joint process, we want to rewrite (\ref{pb_trasp}) in terms of $\Phi$. We have 
 \al{\Es[\|  \xb_\tb- \yb_\tb\|^2] &=\tr \Es[  \xb_\tb\xb_\tb^T +\yb_\tb\yb_\tb^T-\xb_\tb\yb_\tb^T-\yb_\tb\xb_\tb^T]\nn\\
 &=\tr \int_{\Ts^d} (\Phi_\xb +\Phi_\xb-\Phi_{\xb,\yb}-\Phi_{\yb,\xb})\dm}
where $$ \dm (\thb)=\frac{1}{(2\pi)^d }\prod_{k=1}^d \mathrm{d}\tth_k.$$ Then conditions (\ref{cond_marg1}) and (\ref{cond_marg2}) imposes that $\Phi_{\xb}$ and $\Phi_\yb$ are fixed. Accordingly, we obtain the optimal transport problem
\al{ \label{pb_psd}d(p_\xb,p_\yb)^2= &\underset{\Phi_{\xb\yb}  }{\inf}   \tr\int_{\Ts^d} (\Phi_\xb +\Phi_\yb - \Phi_{\xb \yb} -\Phi_{\yb \xb} ) \dm    \nn\\ &    \hbox{ s.t. }  \left[\begin{array}{cc}\Phi_\xb & \Phi_{\xb \yb}\\ \Phi_{\yb\xb} & \Phi_\yb\end{array}\right]   \geq  0 .}
 In what follows, we assume that $\Phi_\xb,\Phi_\yb\in \Sc_m^+(\Ts^d)$ where $\Sc_m^+(\Ts^d)$ denotes the set of multivariate and multidimensional power spectral densities which are bounded and coercive. 
\begin{proposition} \label{prop_1}It holds that
\al{\label{Hell_classical}d(p_\xb,p_\yb)^2=\tr\int_{\Ts^d} (\Phi_\xb+\Phi_\yb-2(\Phi_\yb^{1/2} \Phi_\xb \Phi_\yb^{1/2})^{1/2} )\dm} that is $d(p_\xb,p_\yb)$ is the Hellinger distance between $\Phi_\xb$ and $\Phi_\yb$.  
\end{proposition}
\begin{proof}  It is not difficult to see that (\ref{pb_psd}) is equivalent to solve  \al{ &\underset{\Phi_{\xb\yb}  }{\inf}   -2\tr\int_{\Ts^d}  \Phi_{\xb \yb} \dm    \nn\\ &    \hbox{ s.t. }  \Phi_\xb- \Phi_{\xb \yb}\Phi_\yb^{-1} \Phi_{\yb\xb}   \geq  0. } Then, the proof follows the ideas of one of Proposition 1 in \cite{zorzi2020optimal} for Gaussian stationary processes. The main difference is the fact that here we have multidimensional power spectral densities, while there we have unidimensional power spectral densities.
\end{proof}

In Problem (\ref{pb_trasp}) we can consider a weighted function, that is 
\al{\label{pb_traspW}d_\Omega(p_\xb,p_\yb)^2= \underset{p_{\xb,\yb}\in \Pc}{\inf} \{\Es[\| \hb\star( \xb- \yb)_\tb\|^2]  \hbox{ s.t. (\ref{cond_marg1})-(\ref{cond_marg2}) hold}  \} }
where $\hb:=\{H_\tb,; \; \tb\in\Zs^d\}$, $H_\tb\in \Rs^{m\times m}$ and such that $\hb\in \ell_1^{m\times m}(\Zs^d)$. Then, the latter admits the multidimensional Fourier transform, 
$$ H(e^{j\thb})=\sum_{\tb\in\Zs^d} H_\tb e^{-j\ip{\thb}{\tb}}, \; \; \thb\in\Ts^d.$$ 
In plain words, in (\ref{pb_traspW}) we consider as cost the variance of random field which is obtained by filtering through $\hb$ the discrepancy random field. It is not difficult to see \al{\Es &[\|  \hb\star( \xb- \yb)_\tb\|^2]\nn\\  &=\tr \Es[  (\hb\star \xb)_\tb(\hb\star\xb)_\tb^T +(\hb\star\yb)_\tb(\hb\star\yb)_\tb^T\nn \\ &\hspace{0.3cm}-(\hb\star\xb)_\tb(\hb\star\yb)_\tb^T-(\hb\star \yb)_\tb(\hb\star\xb)_\tb^T]\nn\\
 &=\tr \int_{\Ts^d} \Omega(\Phi_\xb +\Phi_\xb-\Phi_{\xb,\yb}-\Phi_{\yb,\xb})\dm}
 where $\Omega(e^{j\thb})=H(e^{j\thb})H(e^{j\thb})^*$. Accordingly, (\ref{pb_traspW}) is equivalent to solve 
 \al{ &\underset{\Phi_{\xb\yb}  }{\inf}   -2\tr\int_{\Ts^d}  \Omega \Phi_{\xb \yb} \dm    \nn\\ &    \hbox{ s.t. }  \Phi_\xb- \Phi_{\xb \yb}\Phi_\yb^{-1} \Phi_{\yb\xb}   \geq  0. }
 
\begin{proposition} It holds that
\al{\label{Hellinger_w}d_\Omega&(p_\xb,p_\yb)^2\nn\\ &=\tr\int_{\Ts^d} (\Omega \Phi_\xb+\Omega \Phi_\yb-2(\Phi_\yb^{1/2} \Omega\Phi_\xb \Omega\Phi_\yb^{1/2})^{1/2} )\dm} that is $d_\Omega(p_\xb,p_\yb)$ is the weighted Hellinger distance between $\Phi_\xb$ and $\Phi_\yb$ with weight function $\Omega$.  
\end{proposition}
\begin{proof}  The proof is similar to the one of Proposition \ref{prop_1}.
\end{proof}

 \section{OTP between periodic random fields}\label{sec_2}
Consider two jointly Gaussian stationary periodic random fields $\xb=
\{\xb_\tb, \; \tb\in \Zs^d\}$ and $\yb=\{\yb_\tb,\; \tb\in\Zs^d\}$ having zero mean, taking values in $\Rs^m$ and with period $\Nb=(N_1,N_2,\ldots , N_d)$.  This means that for any $\tb=(t_1,t_2,\ldots t_d)$ we have
$$ \yb(\tb)=\yb(t_1,\ldots, t_{l-1},t_l+N_l,t_{l+1},\ldots t_d)$$
almost surely for any $l=1\ldots d$. Accordingly, these random fields are completely characterized by the finite dimensional probability density functions 
$$p_\xb(\xb_\tb,\xb_\ssb; \tb, \ssb), \; \; p_\yb(\yb_\tb,\yb_\ssb; \tb, \ssb)$$
 with $\tb,\ssb\in\Zs^d_\Nb$  and 
 $$\Zs^d_\Nb:=\left\{\tb=(t_1,t_2,\ldots,t_d), \; \; 0\leq t_l\leq N_l-1, \; \;  l=1\ldots d\right\}.$$  The corresponding joint random field is completely characterized by  
$$p_{\xb,\yb}(\xb_\tb,\xb_\ssb,\yb_\ub,\yb_\vb; \tb, \ssb,\ub,\vb)$$ with $\tb,\ssb,\ub,\vb\in\Zs^d_N$.

We consider the following optimal transport problem 
\al{\label{pb_trasp_per}d(p_\xb,p_\yb)^2= \underset{p_{\xb,\yb}\in \Pc}{\inf} \{\Es[\|  \xb_\tb- \yb_\tb\|^2]  \hbox{ s.t. (\ref{cond_marg1p})-(\ref{cond_marg2p}) hold}  \} }
where  
 \al{\label{cond_marg1p}\int_{\Rs^m}\int_{\Rs^m} p_{\xb,\yb}(\xb_\tb, \xb_\ssb,\yb_\ub, & \yb_\vb;\tb,  \ssb, \ub,\vb) \mathrm d \yb_\ub \mathrm d \yb_\vb\nn\\ &= p_{\xb}(\xb_\tb, \xb_\ssb;\tb,\ssb), \; \; \tb,\ssb\in\mathbb Z_\Nb^d\\
\label{cond_marg2p} \int_{\Rs^m}\int_{\Rs^m} p_{\xb,\yb}(\xb_\tb,\xb_\ssb,\yb_\ub , & \yb_\vb; \tb, \ssb,\ub,\vb)\mathrm d \xb_\tb \mathrm d \xb_\ssb\nn\\ &= p_{\yb}(\yb_\ub,\yb_\vb;\ub,\vb),\; \; \ub,\vb\in\mathbb Z^d_\Nb}   
and $\Pc$ is the set of Gaussian joint probability densities $p_{\xb,\yb}$.

Since the joint random field is Gaussian, it is completely characterized by its covariance field
$$ R_\tb=R_{-\tb}^T:=\Es\left[\left[\begin{array}{c}\xb_{\tb+\ssb} \\ \yb_{\tb+\ssb}\end{array}\right] \left[\begin{array}{cc}\xb_\ssb^T & \yb_\ssb^T\end{array}\right]\right], \; \; \tb\in\Zs^d$$
which is also periodic, that is 
$$ R_\tb=R_{(t_1,t_2,\ldots,t_{l-1},t_l+N_l,t_{l+1},\ldots, t_d)}$$
for any $l=1\ldots d$.  Accordingly, its power spectral density is
\al{\label{primadef_psd_per}\Phi(\zb_{\lb}):=\sum_{\tb\in\Zs^d_\Nb} R_{\tb} \zb_{\lb}^{-\tb}  \; \; }
where $\zb=(\zeta_{\ell_1},\zeta_{\ell_2},\ldots, \zeta_{\ell_d})$, $\lb=(l_1,l_2,\ldots,l_d)\in \Zs_\Nb^d$, $\zb_{\lb}^{-\tb}=\prod_{i=1}^d \xi_{\ell_i}^{-t_i}$ and $\xi_{\ell_i}=e^{\frac{2\pi}{N_i}\ell_i}$. Thus, (\ref{primadef_psd_per}) is defined on a discretized $d$-torus and it represents the power spectral density of the joint process.  Also in this case we partition $\Phi(\zb_{\lb})$ according to $\xb$ and $\yb$:
$$ \Phi(\zb_{\lb})= \left[\begin{array}{cc} \Phi_{\xb}(\zb_{\lb}) &  \Phi_{\xb,\yb}(\zb_{\lb}) \\ \Phi_{\yb,\xb}(\zb_{\lb}) & \Phi_{\yb}(\zb_{\lb}) \end{array}\right]$$ and $\Phi_{\xb}$ and $\Phi_{\yb}$ are the power spectral densities of $\xb$ and $\yb$, respectively. Moreover, 
 \al{\Es & [\|  \xb_\tb- \yb_\tb\|^2] =\tr\Es[  \xb_\tb\xb_\tb^T +\yb_\tb\yb_\tb^T-\xb_\tb\yb_\tb^T-\yb_\tb\xb_\tb^T]\nn\\
 &= \frac{1}{|\Nb|}\tr \sum_{\lb \in\Zs_\Nb^d}  (\Phi_\xb(\zb_{\lb}) +\Phi_\xb(\zb_{\lb})-\Phi_{\xb,\yb}(\zb_{\lb})-\Phi_{\yb,\xb}(\zb_{\lb}))}
 where $|\Nb|:=\prod_{l=1}^d N_l$. Accordingly, the optimal transport problem in (\ref{pb_trasp_per}) is equivalent to 
\al{ \label{pb_psd_per}d&(p_\xb,p_\yb)^2\nn\\&= \underset{\Phi_{\xb\yb}  }{\inf}   \frac{1}{|\Nb|}\tr \sum_{\lb \in\Zs_\Nb^d}  (\Phi_\xb(\zb_{\lb}) +\Phi_\xb(\zb_{\lb})-\Phi_{\xb,\yb}(\zb_{\lb})-\Phi_{\yb,\xb}(\zb_{\lb}) )  \nn\\ &    \hbox{ s.t. } \Phi_\xb(\zb_{\lb})- \Phi_{\xb \yb}(\zb_{\lb})\Phi_\yb(\zb_{\lb})^{-1} \Phi_{\yb\xb} (\zb_{\lb})    \geq  0, \; \; \forall \, \lb\in\Zs_\Nb^d }
 where we assumed that $\Phi_\xb (\zb_{\lb})>0$ and $\Phi_\yb (\zb_{\lb})>0$ for any $\lb\in\Zs_\Nb^d$.
\begin{proposition} \label{prop_1_per}It holds that
\al{d(p_\xb,p_\yb)^2= \frac{1}{|\Nb|}&\tr \sum_{\lb \in\Zs_\Nb^d} (\Phi_\xb(\zb_{\lb})+\Phi_\yb(\zb_{\lb})\nn\\ &-2(\Phi_\yb(\zb_{\lb})^{1/2} \Phi_\xb(\zb_{\lb}) \Phi_\yb(\zb_{\lb})^{1/2})^{1/2} )\nn} that is $d(p_\xb,p_\yb)$ is the Hellinger distance between $\Phi_\xb$ and $\Phi_\yb$.  
\end{proposition}
\begin{proof} The proof is similar to the one of Proposition \ref{prop_1}.
\end{proof}

Similarly to the aperiodic case, we can generalize Problem (\ref{pb_trasp_per}) by considering a periodic weight function 
$\hb:=\{H_\tb,; \; \tb\in\Zs^d\}$, $H_\tb\in \Rs^{m\times m}$,  with period $\Nb$, that is
$$H_\tb=H_{(t_1,t_2,\ldots, t_{l-1},t_l+N_l,t_{l+1},\ldots ,t_d)}$$
for any $l=1\ldots d$. The corresponding multidimensional Fourier transform is 
$$ H(\zb_{\lb})=\sum_{\tb\in\Zs^d_\Nb} H_\tb \zb_{\lb}^{-\tb}, \; \;  \lb \in\Zs^d_\Nb.$$ 
Thus, we consider
\al{\label{pb_traspWper}d_\Omega(p_\xb,p_\yb)^2= \underset{p_{\xb,\yb}\in \Pc}{\inf} \{\Es[\| \hb\ast( \xb- \yb)_\tb\|^2]  \hbox{ s.t. (\ref{cond_marg1p})-(\ref{cond_marg2p}) hold}  \} }
where the symbol $\ast$ denotes the circular discrete convolution, that is 
$$ \left(\hb\ast x\right)_\tb:=\sum_{\ssb\in\Zs_\Nb^d} \hb_{\tb-\ssb} \xb_\ssb .$$ Now, the cost function is the variance of the periodic random field which is obtained by filtering through $\hb$ the discrepancy random field. Accordingly, we have\al{\Es &[\|  \hb\ast( \xb- \yb)_\tb\|^2]\nn\\  &=\tr \Es[  
 (\hb\ast \xb)_\tb(\hb\ast\xb)_\tb^T +(\hb\ast\yb)_\tb(\hb\ast\yb)_\tb^T\nn \\ &\hspace{0.3cm}-(\hb\ast\xb)_\tb(\hb\ast\yb)_\tb^T-(\hb\ast \yb)_\tb(\hb\ast\xb)_\tb^T]\nn\\
 &=\frac{1}{|\Nb|}\tr\sum_{\lb\in\Zs_\Nb^d}  \Omega(\zb_{\lb})[\Phi_\xb(\zb_{\lb}) +\Phi_\xb(\zb_{\lb})\nn\\ &\hspace{0.3cm}-\Phi_{\xb,\yb}(\zb_{\lb})-\Phi_{\yb,\xb}(\zb_{\lb})]\nn}
 where $\Omega(\zb_{\lb})=H(\zb_{\lb})H(\zb_{\lb})^*$.  
\begin{proposition} It holds that
\al{\label{Hell_w_per}d_\Omega(p_\xb,&p_\yb)^2=\frac{1}{|\Nb|}\tr\sum_{\lb\in\Zs_\Nb^d}(\Omega(\zb_{\lb})\Phi_\xb(\zb_{\lb})+\Omega(\zb_{\lb})\Phi_\yb(\zb_{\lb})\nn\\ &-2(\Phi_\yb^{1/2} \Omega(\zb_{\lb})\Phi_\xb(\zb_{\lb}) \Omega(\zb_{\lb})\Phi_\yb(\zb_{\lb})^{1/2})^{1/2} ) } that is $d(p_\xb,p_\yb)$ is the weighted Hellinger distance between $\Phi_\xb$ and $\Phi_\yb$ with weight function $\Omega$.  
\end{proposition}
\begin{proof}  The proof is similar to the one of Proposition \ref{prop_1}.
\end{proof}

 \section{Spectral morphing}\label{sec_3}
 Consider a zero mean Gaussian (aperiodic) random field whose description slowly varies over time. Moreover, suppose that in a sufficiently small time interval $[k-\sigma,k+\sigma]$, for some $\sigma\in\Ns$, the random field can be considered to be stationary. Therefore, at time $k$ it can be approximately described by a power spectral density, say $\Phi_k(e^{j\thb})$. It is then natural to wonder how to construct a smooth interpolation between nearby power spectral densities, e.g.  $\Phi_{k-1}(e^{j\thb})$ and $\Phi_k(e^{j\thb})$. The latter task is referred to as spectral morphing. A possible smooth interpolation is given by the geodesic defined by the weighted Hellinger distance (\ref{Hellinger_w}) on the manifold of the   multivariate and multidimensional power spectral densities. For simplicity, consider the nearby spectral densities at $k=0$ and $k=1$, then we have 
 \al{ \label{H_W_eucl}d_\Omega& (\Phi_0,\Phi_1)^2\nn\\ &=\tr\int_{\Ts^d} (\Omega\Phi_0+\Phi_1-2(\Phi_0^{1/2} \Omega\Phi_1 \Omega\Phi_0^{1/2})^{1/2} )\dm\nn\\
 &=\tr\int_{\Ts^d}\| \Phi_0^{1/2}-\Phi_1^{1/2} U_\Omega\|^2_{\Omega}\dm} where
 \al{U_\Omega&(e^{j\thb})=\Phi_1^{-1/2}(e^{j\thb})\Omega^{-1}(e^{j\thb}) \Phi_0^{-1/2}(e^{j\thb})\nn\\ &\times (\Phi_0^{1/2}(e^{j\thb})\Omega(e^{j\thb})\Phi_1(e^{j\thb})\Omega(e^{j\thb}) \Phi_0^{1/2}(e^{j\thb}))^{1/2}}
 and $U_\Omega(e^{j\thb})U_\Omega(e^{j\thb})^*=I$, i.e. $U_\Omega$ is an all-pass function. In view of (\ref{H_W_eucl}), $d_\Omega(\Phi_0,\Phi_1)$ is the weighted Euclidean distance between the spectral factors $\Phi_0^{1/2}(e^{j\thb})$ and $\Phi_1^{1/2}(e^{j\thb})U_\Omega(e^{j\thb})$ and thus the corresponding geodesic is the line segment connecting them. Accordingly, the geodesic on the manifold of the multivariate and multidimensional   power spectral densities connecting $\Phi_0(e^{j\thb})$ and $\Phi_1(e^{j\thb})$ is 
 \al{ \label{geod}\Phi_\tau&(e^{j\thb})=[(1-\tau)\Phi_0(e^{j\thb})^{1/2}+\tau \Phi_1(e^{j\thb})^{1/2}U_\Omega(e^{j\thb})]\nn\\ &\times[(1-\tau)\Phi_0(e^{j\thb})^{1/2}+\tau \Phi_1(e^{j\thb})^{1/2}U_\Omega(e^{j\thb})]^*}
 with $\tau\in [0,1]$.
In the special case that $\Omega (e^{j\thb})=I$, i.e. when we consider the Hellinger distance in (\ref{Hell_classical}), the all-pass function used to form the geodesic becomes 
\al{U_I(e^{j\thb})=&\Phi_1^{-1/2}(e^{j\thb})\Phi_0^{-1/2}(e^{j\thb})\nn\\ &\times (\Phi_0^{1/2}(e^{j\thb})\Phi_1(e^{j\thb})\Phi_0^{1/2}(e^{j\thb}))^{1/2}}
which is the one considered in \cite{NING_2013}. It is also worth noting that in the case that $m=1$, i.e. we consider the manifold of the univariate and multidimensional spectral densities, then $U_\Omega(e^{j\thb})=U_I(e^{j\thb})$ that is (\ref{Hellinger_w}) and (\ref{Hell_classical}) define the same geodesic.

In the periodic case, the weighted Hellinger distance in (\ref{Hell_w_per}) 
defines the following geodesic on the manifold of the multivariate multidimensional power spectral densities:
\al{ \Phi_\tau&(\zb_{\lb})=[(1-\tau)\Phi_0(\zb_{\lb})^{1/2}+\tau \Phi_1(\zb_{\lb})^{1/2}U_\Omega(\zb_{\lb})]\nn\\ &\times[(1-\tau)\Phi_0(\zb_{\lb})^{1/2}+\tau \Phi_1(\zb_{\lb})^{1/2}U_\Omega(\zb_{\lb})]^*}
 with $\tau\in [0,1]$ and $U_\Omega(\zb_{\lb})$ is an all-pass function, i.e. $U_\Omega(\zb_{\lb})U_\Omega(\zb_{\lb})^*=I$ for any $\lb\in\Zs_\Nb^d$, defined as follows:
 \al{U_\Omega&(\zb_{\lb})=\Phi_1^{-1/2}(\zb_{\lb})\Omega^{-1}(\zb_{\lb}) \Phi_0^{-1/2}(\zb_{\lb})\nn\\ &\times (\Phi_0^{1/2}(\zb_{\lb})\Omega(e^{j\thb})\Phi_1(\zb_{\lb})\Omega(\zb_{\lb})\Phi_0^{1/2}(\zb_{\lb}))^{1/2}.}
 \section{Example}\label{sec_4}
 We consider two zero mean Gaussian random fields with $d=2$ and $m=2$ having spectral density $\Phi_0(e^{j\thb})=W_1(e^{j\thb})W_1(e^{j\thb})^*$ and $\Phi_1(e^{j\thb})=W_2(e^{j\thb})W_2(e^{j\thb})^*$, respectively. More precisely, 
 \al{W_0(\zzb)&=\left[\begin{array}{cc} \frac{1}{1-\ip{\ab_0}{\zzb^{-1}}} &  \frac{1}{1-\ip{\bb_0}{\zzb^{-1}}}  \\0 & \frac{1}{1-\ip{\gb_0}{\zzb^{-1}}}  \end{array}\right]\nn\\
 W_1(\zzb)&= \left[\begin{array}{cc} \frac{1}{1-\ip{\ab_1}{\zzb^{-1}}}  &  \frac{1}{1-\ip{\bb_1}{\zzb^{-1}}} \\0 &\frac{1}{1-\ip{\gb_1}{\zzb^{-1}}} \end{array}\right]}
 where $\rho=0.475$, \al{\ab_0&=\rho[\,e^{jpi/2} \; e^{j\pi/2}\,]^T, \; \; \ab_1 =\rho[\,e^{j3\pi/4} \; e^{j\pi/2}\,]^T, \nn\\ \bb_0&=\rho[\, e^{j\pi/3}\; e^{j3\pi/4}\,]^T,\; \; \bb_1=\rho[\, e^{j\pi/2}\; e^{j3\pi/4}\,]^T,\nn\\ \gb_0&=\rho[\, e^{j3\pi/4}\; e^{j\pi/3}\,]^T,  \; \; \gb_1=\rho[\,e^{j3\pi/4} \;e^{j\pi} \,]^T}
 and, with some abuse of notation, $\ip{\bar\ab}{\zzb^{-1}}:=\bar\alpha_1z_1^{-1}+\bar\alpha_1z_2^{-1}$ with $\bar\ab=[\,\bar\alpha_1\;\bar\alpha_2\,]^T$ and $\zzb=[\,z_1\;z_2\,]^T$.

Figure \ref{fig} shows the corresponding geodesic defined in (\ref{geod}) with the constant weight function
\al{\label{def_OM} \Omega(e^{j\thb})=\left[\begin{array}{ccc}1 & -0.99   \\ -0.99 & 1  \end{array}\right]} for $\tau=0$ (first row), $\tau=0.33$ (third row), $\tau=0.67$ (fifth row) and $\tau=1$ (sixth row). Moreover, we also compare it with the geodesic obtained with $\Omega(e^{j\thb})=I$ for $\tau=0$ (first row), $\tau=0.33$ (second row), $\tau=0.67$ (fourth row) and $\tau=1$ (sixth row). We can notice that the two geodesics are visibly different in regard to the real part of the entry in position $(1,2)$. Accordingly, we can design $\Omega $ in such a way to induce specific properties on the corresponding geodesic.

 \begin{figure*}
 \includegraphics[width=\textwidth]{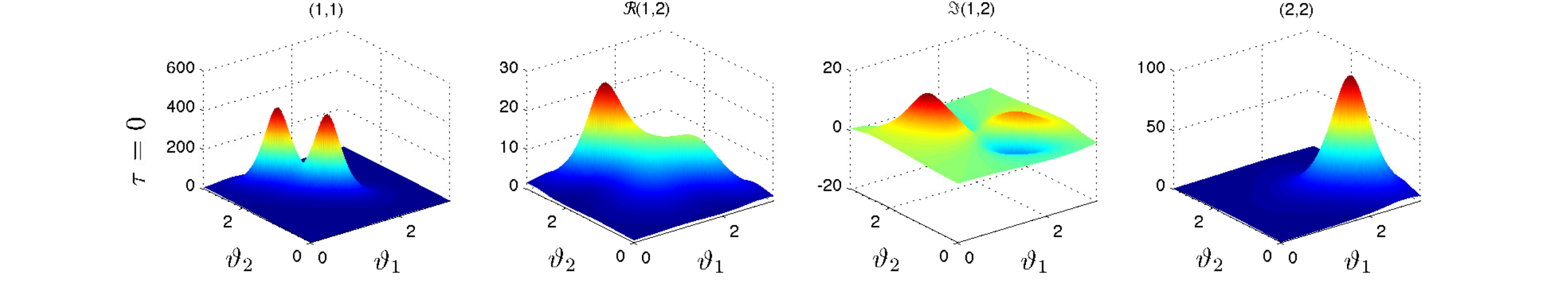}
  \includegraphics[width=\textwidth]{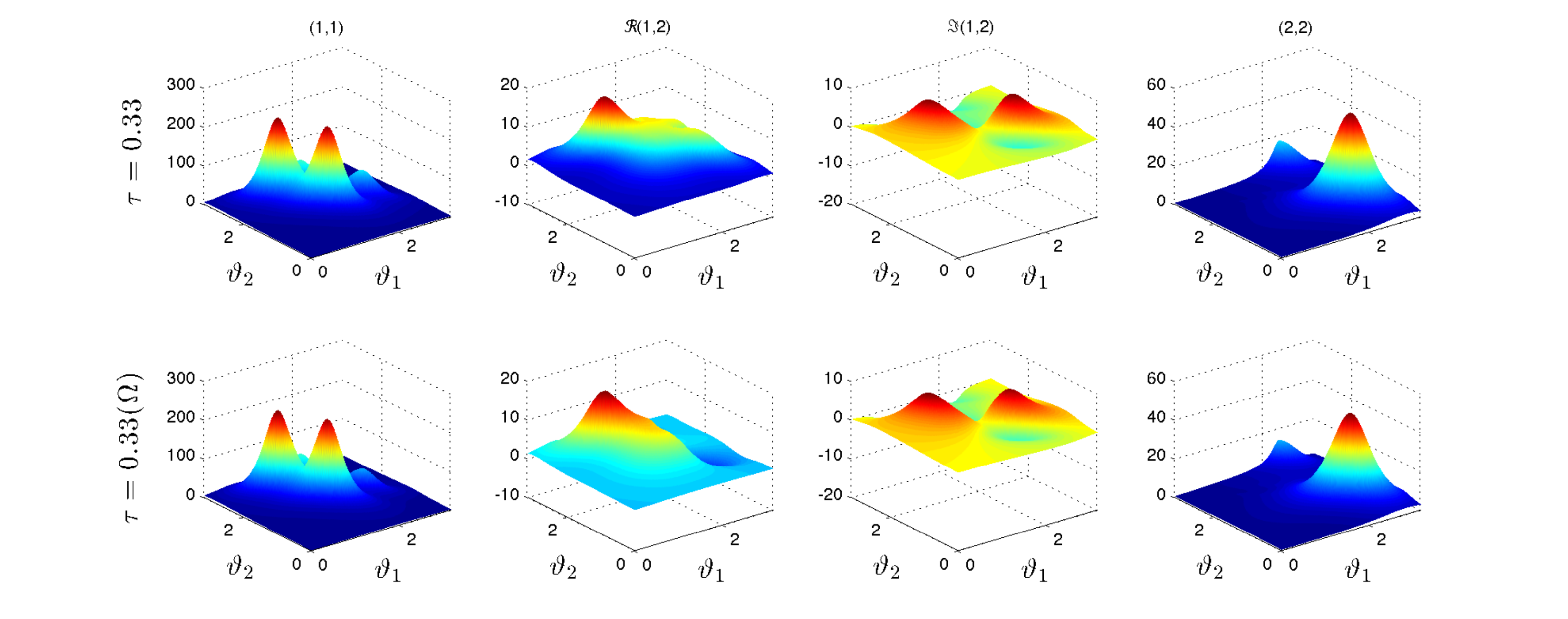}
  \includegraphics[width=\textwidth]{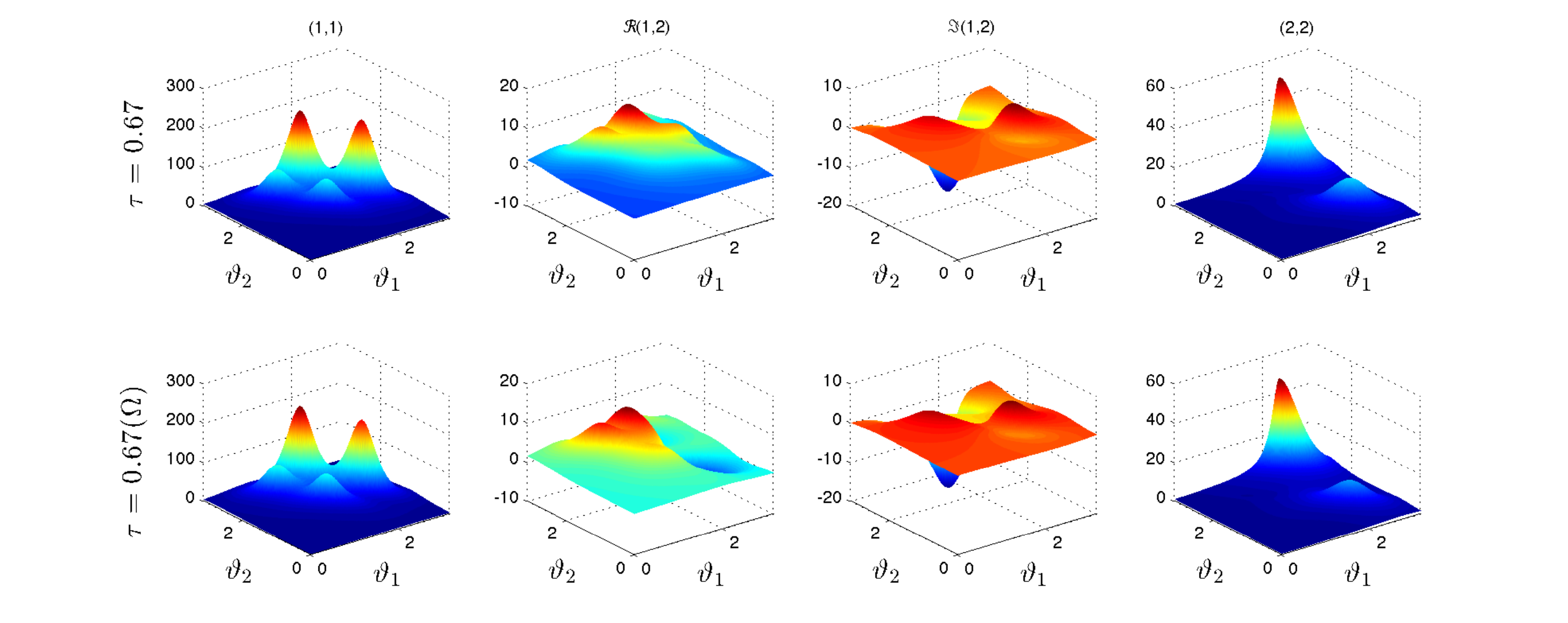}
      \includegraphics[width=\textwidth]{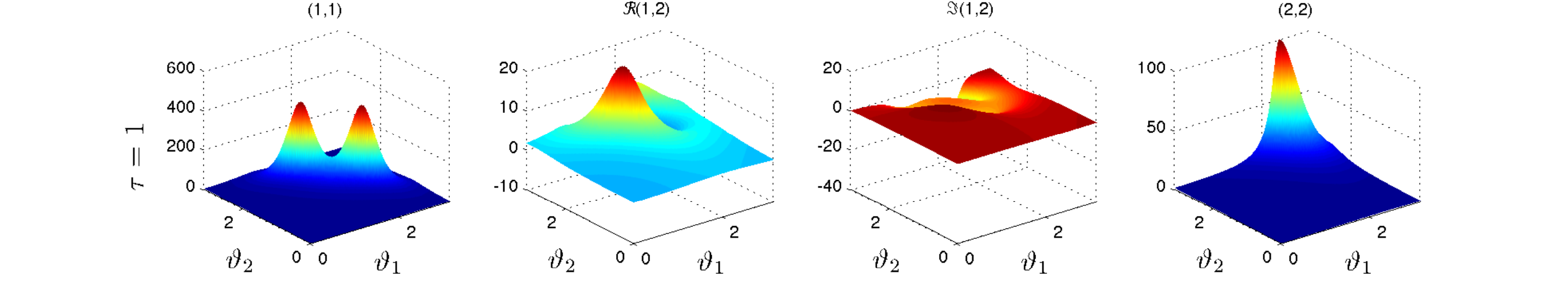}
   \caption{The path $\Phi_\tau(e^{j\thb})$ between $\Phi_0(e^{j\thb})$ and $\Phi_1(e^{j\thb})$ for $\tau\in\{0, 0.33, 0.67, 1\}$ using $\Omega$ defined in (\ref{def_OM}) -- rows one, three, five and six -- and $\Omega(e^{j\thb})=I$ -- rows one, two, four and six. The first and the last column show the entry of the spectral densities in position $(1,1)$ and $(2,2)$, respectively. The second and the third column show the real and the imaginary part of the entry of the spectral densities  in position $(1,2)$.}\label{fig}
 \end{figure*}
 \section{The general case} \label{sec_4b}
 The OTP's analyzed before consider $\Pc$
 as the set of Gaussian joint probability densities. This hypothesis, however, can be weakened. Notice that a Gaussian process is a particular elliptical process. More precisely, 
we can take $\Pc$ as the set of the joint probability densities such that $[\,\xb^T \; \yb^T\,]^T$ is an elliptical stationary process having zero mean and with joint power spectral density bounded and coercive. Accordingly, $\xb$ and $\yb$ are elliptical processes with zero mean. We conclude that the same reasoning and thus same results hold also in this case.
 
 \section{Conclusion}\label{sec_5}
In this paper we have introduced the optimal transport problem between Gaussian aperiodic/periodic Gaussian random fields. The solution to these problems leads to a weighted Hellinger distance between multivariate and multidimensional power spectral densities. Such a distance can be characterized in terms of spectral factors. In the unidimensional case, the Hellinger distance can be defined in such a way to have the freedom in choosing one  of these two spectral factors, see \cite{Hellinger_Ferrante_Pavon}; in particular, it is always possible to choose  a rational spectral factor if the corresponding spectral density is rational. It is worth stressing that this last fact in the multidimensional case, however, is no longer true in general,
\cite{geronimo2006factorization,geronimo2004positive}.

Finally, we have shown that the weighted Hellinger distance defines a geodesic, depending on the weight function, on the manifold of the multivariate and multidimensional spectral densities.

\end{document}